\documentclass[a4paper, 12pt]{article}
\usepackage{amsmath}
\usepackage{amsfonts}
\usepackage{amsthm}
\usepackage{graphicx}
\usepackage{hyperref}

\newtheorem*{teo}{Theorem}
\newtheorem{lem}{Lemma}
\newtheorem{prop}{Proposition}

\newtheorem{defin}{Definition}
\newtheorem{teo1}{Theorem}
\newtheorem*{teoA}{Theorem A}
\newtheorem*{teoB}{Theorem B}
\newtheorem*{teoC}{Theorem C}

\newcommand{\ric}{\operatorname{Ric}}

\newcommand{\diver}{\operatorname{div}}
\newcommand{\hess}{\operatorname{Hess}}
\newcommand{\dist}{\operatorname{dist}}

\newcommand{\inj}{\operatorname{inj}}
\newcommand{\Cut}{\operatorname{Cut}}
\newcommand{\lp}{\operatorname{L}}

\title{Parabolic Minimal Surfaces in $\mathbb{M}^{2}\times\mathbb{R}$}
\author{Vanderson Lima\thanks{The author was supported by CNPq-Brazil.}}
\date{}

\begin{document}

\maketitle

\begin{abstract}
\noindent Let $\mathbb{M}^{2}$ be a complete non compact orientable surface of non negative curvature. We prove in this paper some theorems involving parabolicity of minimal surfaces in $\mathbb{M}^{2}\times\mathbb{R}$. 

First, using a characterization of $\delta$-parabolicity we prove that under additional conditions on $\mathbb{M}$, an embedded minimal surface with bounded gaussian curvature is proper. 

The second theorem states that under some conditions on $\mathbb{M}$, if $\Sigma$ is a properly immersed minimal surface with finite topology and one end in $\mathbb{M}\times\mathbb{R}$, which is transverse to a slice $\mathbb{M}\times\{t\}$ except at a finite number of points, and such that $\Sigma\cap(\mathbb{M}\times\{t\})$ contains a finite number of components, then $\Sigma$ is parabolic. 

In the last result, we assume some conditions on $\mathbb{M}$ and prove that if a minimal surface in $\mathbb{M}\times\mathbb{R}$ has height controlled by a logarithmic function, then it is parabolic and has a finite number of ends.
\end{abstract}

\providecommand{\abs}[1]{\lvert#1\rvert}

\linespread{1} 

\section{Introduction}

Let $\mathbb{M}^{2}$ be a complete non compact orientable surface with non-negative curvature. Under these conditions $\mathbb{M}\times\mathbb{R}$ is complete and has non-negative sectional curvature, in particular non-negative Ricci curvature. Recently, using some of the results of \cite{S.Y2}, G.Liu classified complete non-compact $3$-manifolds with non-negative Ricci curvature 
\begin{teo}[Liu, \cite{L}]
Let $N$ be a complete noncompact $3$-manifold with non-negative Ricci curvature. Then either $N$ is diffeomorphic to $\mathbb{R}^{3}$ or its universal cover $\widetilde{N}$ is isometric to a Riemannian product $\mathbb{M}\times\mathbb{R}$, where $\mathbb{M}$ is a complete surface with non-negative sectional curvature.
\end{teo}
In particular it follows from the proof of this result that if $N$ is not flat or does not have positive Ricci curvature then its universal cover splits as a product $\mathbb{M}\times\mathbb{R}$. So the spaces $\mathbb{M}\times\mathbb{R}$ are in fact general examples of a very important class of $3$-manifolds. 

We are interested in minimal surfaces in $\mathbb{M}\times\mathbb{R}$, where $\mathbb{M}$ is as above. In particular we want information about the topology and the conformal structure. It is important to study under which hypotheses we can guarantee that a minimal surface is proper. Concerning the topology, we know that there is no compact minimal surface in these spaces. So, one can study the genus and the number of ends of such minimal surfaces. Concerning the conformal structure, one important property is {\it parabolicity}. Our results are inspired by analogous results in $\mathbb{R}^{3}$. 

First we study the problem of properness.  Bessa, Jorge and Oliveira-Filho studied
this problem for manifolds with nonnegative Ricci curvature and obtained some
partial results in $\mathbb{R}^{3}$.
\begin{teo}[Bessa, Jorge, Oliveira-Filho, \cite{B.J.O}]
Let $N^{3}$ be a complete Riemannian $3$-manifold of bounded geometry and positive Ricci curvature. Let $f:\Sigma^{2} \to N^{3}$ be a complete injective minimal immersion, where $\Sigma$ is a complete oriented surface with bounded curvature. 
\begin{enumerate}
\item If $N$ is compact, then $\Sigma$ is compact;
\item If $N$ is not compact, then $f$ is proper.
\end{enumerate}
\end{teo}

A major breakthrough was the work of Colding and Minicozzi \cite{C.M2}, where it was proved that a complete minimal surface of finite topology embedded in $\mathbb{R}^{3}$ is proper. After this, Meeks and Rosenberg proved that if $\Sigma$ is a complete embedded minimal surface in $\mathbb{R}^{3}$ which has positive injectivity radius, then $\Sigma$ is proper, \cite{M.R2}. Finally, Meeks and Rosenberg proved that if $f: \Sigma \to \mathbb{R}^{3}$ is an injective minimal immersion, with $\Sigma$ complete and of bounded curvature, then $f$ is proper, \cite{M.R1}. We extend the last result to the case of a product $\mathbb{M}\times\mathbb{R}$:
\begin{teoA}
Let $\mathbb{M}$ be a complete simply connected orientable non-compact surface such that $0 \leq K_{\mathbb{M}} \leq \kappa$. Let $f: \Sigma \rightarrow \mathbb{M}\times\mathbb{R}$ be an injective minimal immersion of a complete, connected Riemannian surface of bounded curvature. Then the map $f$ is proper.
\end{teoA}

Next we focus on surfaces with finite topology and one end. The results in \cite{C.M2} and \cite{M.R} imply that every complete, embedded minimal surface in $\mathbb{R}^{3}$ of finite genus and one end is properly embedded and intersects some plane transversely in a single component, and so, is parabolic. In \cite{M.R1}, Meeks and Rosenberg gave an independent proof that the surface is parabolic without the additional assumption that it is embedded. Namely they proved:  
\begin{teo}[Meeks, Rosenberg, \cite{M.R1}]\label{mrp}
Let $\Sigma$ be a surface of finite topology and one end, and let $f: \Sigma \rightarrow \mathbb{R}^{3}$ be a proper minimal immersion. Suppose that $f$ is transverse to a plane $P$ except at a finite number of points, and $f^{-1}(P)$ contains a finite number of components. Then, $\Sigma$ is parabolic.
\end{teo}
The Half-Space Theorem of Hoffman and Meeks states that a properly immersed minimal surface in $\mathbb{R}^{3}$ which is above a plane is a parallel plane, \cite{H.M}. Thus the condition that a minimal surface be transverse to a plane is natural. Rosenberg proved the following Half-Space Theorem for product spaces:
\begin{teo}[Rosenberg, \cite{H}]
Let $\mathbb{M}$ be a complete non-compact surface satisfying the following conditions:
\begin{enumerate}
\item $K_{\mathbb{M}} \geq 0$; 
\item There is a point $p \in M$ such that the geodesic curvatures of all geodesic circles with center $p$ and radius $r \geq 1$ are uniformly bounded.
\end{enumerate}
If $\Sigma$ is a properly immersed minimal surface in a half-space $\mathbb{M}\times [t_0,+\infty)$, then $\Sigma$ is a slice $\mathbb{M}\times\{s\}$, for some $s > t_0$.

\end{teo}

Based on these results we prove the following:

\begin{teoB}
Suppose $\mathbb{M}$ satisfies the conditions of the previous theorem. Let $\Sigma$ be a surface of finite topology and one end and let $f: \Sigma \rightarrow \mathbb{M}\times\mathbb{R}$ be a proper minimal immersion. Suppose that $f$ is transverse to a slice $\mathbb{M}\times\{t_0\}$ except at a finite number of points, and that $f^{-1}(\mathbb{M}\times\{t_0\})$ contains a finite number of components. Then $\Sigma$ is parabolic.
\end{teoB}

Next we focus on surfaces with more that one end. A major breakthrough was the proof of the generalized Nitsche conjecture in $\mathbb{R}^{3}$:
\begin{teo}[Collin, \cite{C}]
Let $\Sigma$ be a properly embedded minimal surface in $\mathbb{R}^{3}$ with at least two ends. Then an annular end of $\Sigma$ is asymptotic to a plane or to the end of a catenoid.
\end{teo}
Let $\Sigma$ be as in the last theorem. The set $\mathcal{E}_{\Sigma}$ of all the ends of $\Sigma$ has a natural topology that makes it a compact Hausdorff space. The limit points in $\mathcal{E}_{\Sigma}$ are called the {\it limit ends} of $\Sigma$ and an end which is not a limit end is called a {\it simple end}. To $\Sigma$ is associated a unique plane $P$ passing through the origin in $\mathbb{R}^{3}$ called the limit tangent plane at infinity of $\Sigma$, \cite{C.H.M}. The ends of $\Sigma$ are linearly ordered by their relative heights over $P$ and this linear ordering, up to reversing it, depends only on the proper ambient isotopy class of $\Sigma$ in $\mathbb{R}^{3}$, \cite{F.M}. Since $\mathcal{E}_{\Sigma}$ is compact and the ordering is linear, there exists a unique {\it top end} which is the highest end and a unique {\it bottom end} which is lowest in the associated ordering. The ends of $\Sigma$ that are neither top nor bottom ends are called {\it middle ends}. In the proof of the ordering theorem, one shows that every middle end of $\Sigma$ is contained between two catenoids in the following sense: if $E$ is an end of $\Sigma$ there are $c_{1} > 0$ and $r_1 > 0$ such that $E \subset \{(x_{1},x_{2},x_{3}); |x_{3}| \leq c_{1}\log r, r^{2} = x_{1}^{2} + x_{2}^{2}, r \geq r_{1}\}$. 

In \cite{C.K.M.R} Collin, Kusner, Meeks and Rosenberg proved that if $\Sigma$ is a properly immersed minimal surface with compact boundary in $\mathbb{R}^{3}$ which is contained between two catenoids, then $\Sigma$ has quadratic area growth. Furthermore, $\Sigma$ has a finite number of ends. As a consequence the middle ends of a properly embedded minimal surface in $\mathbb{R}^{3}$ are never {\it limit ends}. We explain what it means for a properly immersed minimal surface of $\mathbb{M}\times\mathbb{R}$ to be contained between two catenoids and generalize the result above:
\begin{teoC}
Let $\mathbb{M}$ be a complete non-compact surface satisfying the following conditions:
\begin{enumerate}
\item $0 \leq K_{\mathbb{M}} \leq \kappa$;
\item $\mathbb{M}$ has a pole $p$; 
\item The geodesic curvatures of all geodesic circles with center $p$ and radius $r \geq 1$ are uniformly bounded.
\end{enumerate} 
Let $\Sigma$ be a properly immersed minimal surface inside the region of $\mathbb{M}\times\mathbb{R}$ defined by $|h|\leq c_{2}\log r$, for some constant $c_{2}>0$ and $r\geq1$. Then $\Sigma$ is parabolic. Moreover, if $\Sigma$ has compact boundary, then $\Sigma$ has quadratic area growth and a finite number of ends. 
\end{teoC}

The paper is organized as follows. In section 2 we present some results about the geometry of the spaces $\mathbb{M}\times\mathbb{R}$ and its minimal surfaces. In sections 3 and 4 we give some well known definitions and enunciate some results involving parabolicity and laminations. In section 5 we prove theorem A. In section 6 we prove theorems B and C.
\\\\
\textbf{Acknowledgments.} This work is part of the author's Ph.D. thesis at IMPA. The
author would like to express his sincere gratitude to his advisor Harold Rosenberg for his patience, constant encouragement and guidance. He would also like to thank Marco A. M. Guaraco for making the figures that appear in the paper, and Beno\^it Daniel, Jos\'e Espinar, L\' ucio Rodriguez, Magdalena Rodriguez and William Meeks III for discussions and their interest in this work. Finally he also thanks the referee for suggestions and corrections.

\section{The Geometry of $\mathbb{M}^{2}\times\mathbb{R}$}

Some of the results of this section are well known, but we prove them here for completeness.

\begin{lem}
Let $\mathbb{M}$ be a complete non-compact orientable surface with non-negative sectional curvature. Then $\mathbb{M}$ is homeomorphic to $\mathbb{R}^{2}$ or isometric to a flat cylinder $\mathbb{S}^{1}\times\mathbb{R}$.
\end{lem}

\begin{proof}
Since $K_{\mathbb{M}}^{-} \equiv 0$, by Huber's Theorem $\mathbb{M}$ has finite topology and $$0 \leq \int_{\mathbb{M}}K_{\mathbb{M}}d\mu \leq 2\pi(2 -2g - n),$$ where $g$ is the genus of $M$ and $n$ his number of ends, see \cite{W1, W2}. Since $\mathbb{M}$ is non-compact and $n \geq 1$, we have $$1 \leq n + 2g \leq 2,$$ but $n + 2g $ is an integer; thus the only possibility is $g = 0$, $n = 1,2$.\\

If $n = 1$, $\mathbb{M}$ is homeomorphic to $\mathbb{R}^{2}$. If $n = 2$, $\mathbb{M}$ has the topology of $\mathbb{S}^{1}\times\mathbb{R}$ and $$\int_{\mathbb{M}}K_{\mathbb{M}}d\mu = 0,$$ thus $K_{\mathbb{M}} \equiv 0$ and $\mathbb{M}$ is isometric to $\mathbb{S}^{1}\times\mathbb{R}$ endowed with a flat metric.
\end{proof}

\begin{lem}
Let $\mathbb{M}$ be a complete non-compact surface with sectional curvature satisfying $0 \leq K_{\mathbb{M}} \leq \kappa$. Then $\mathbb{M}$ has positive injectivity radius, in particular the same holds for $\mathbb{\mathbb{M}}\times\mathbb{R}$.
\end{lem}

\begin{proof}
By the previous lemma either $\mathbb{M}$ is a flat cylinder, which has positive injectivity radius, or $\mathbb{M}$ is homeomorphic to $\mathbb{R}^{2}$. Suppose in the last case that $\inj_{\mathbb{M}} = 0$. Since $K_{\mathbb{M}} \leq \kappa$ the exponential map $\exp_{q}: B_{\frac{\pi}{\sqrt{\kappa}}}(0) \rightarrow \mathbb{M}$ has no critical points for each $q \in \mathbb{M}$. Then for each positive integer $j$ sufficiently large there is a point $p_{j}$ such that $\exp_{p_{j}}$ is not injective in the geodesic ball $B_{1/j}(p_{j})$, which implies there are two geodesics $\gamma_{j},\sigma_{j}: [0,l] \rightarrow \mathbb{M}$ beginning in $p_{j}$ which meet at the same endpoint $q_{j}$ in the boundary of $B_{1/j}(p_{j})$ with angle equal to $\pi$ $\bigl(q_{j}$ realizes the distance from $p_{j}$ to $\Cut(p_{j})$; see \cite{MC}$\bigr)$. This gives us a geodesic loop $\alpha_{j}$ with one angular vertex at $p_{j}$ which has exterior angle $\theta_{j} \leq \pi$. Since $\mathbb{M}$ is simply connected $\alpha_{j}$ bounds a disc $D_{j}$ in $\mathbb{M}$. By the Gauss-Bonnet Theorem
$$2\pi = \int_{D_{j}}K_{\mathbb{M}}\ d\mu + \theta_{j} \leq \kappa|D_{j}| + \pi.$$

However for $j$ sufficiently large $|D_{j}|$ is small and $\kappa|D_{j}| + \pi < 2\pi$, which is a contradiction. Therefore $\inj_{\mathbb{M}} > 0$.  
\end{proof}

\begin{lem}[\cite{E.R}]\label{tot.geod}
Let $\mathbb{M}$ be a complete connected non flat surface. Let $\Sigma$ be a complete totally geodesic surface in $\mathbb{M}\times\mathbb{R}$. Then $\Sigma$ is of the form $\alpha\times\mathbb{R}$, where $\alpha$ is a geodesic of $M$, or $\Sigma = \mathbb{M}\times\{t\}$ for some $t \in \mathbb{R}$.
\end{lem}

\begin{proof} 
Let $\Pi$ be the projection of $\mathbb{M}\times\mathbb{R}$ to $\mathbb{M}$. Let $\eta$ be a unit normal to $\Sigma$ and define $\nu = \langle \eta,\partial_{t}\rangle$. Since $\Sigma$ is totally geodesic we have
\begin{eqnarray}
\label{gauss}K_{\Sigma}(p) &=& K_{\mathbb{M}}\bigl(\Pi(p)\bigr)\nu(p), \forall p \in \Sigma,\\
\label{deriv}X\langle\eta,\partial_{t}\rangle &=& \langle\nabla_{X}\eta,\partial_{t}\rangle \equiv 0, \forall X \in T\Sigma
\end{eqnarray}
where \eqref{gauss} is just the Gauss equation. So $\nu$ is constant, and we can suppose $\nu \geq 0$. If $\nu = 0$, then $\Sigma$ is of the form $\alpha\times\mathbb{R}$. If $\nu = 1$, then $\Sigma$ is a slice.

Suppose $0 < \nu < 1$. We know that $$\Delta_{\Sigma}\nu + \bigl(\ric(\eta,\eta) + |A|^2\bigr)\nu = 0,$$ 
and by equation \eqref{deriv}, $\Delta_{\Sigma}\nu = 0$. Thus $0 = \ric(\eta,\eta) = K_{\mathbb{M}}\bigl(\Pi(p)\bigr)(1 - \nu^2)$ which implies $K_{\mathbb{M}}\bigl(\Pi(p)\bigr) = 0$. It follows from equation \eqref{gauss} that $\Sigma$ is flat. Then there is a $\delta > 0$ such that, for any $p \in \Sigma$ a neighborhood of $p$ in $\Sigma$ is a graph (in exponential coordinates) over the disc $D_{\delta} \subset T_{p}\Sigma$ of radius $\delta$, centered at the origin of $T_{p}\Sigma$. This graph, denoted by $G_{p}$, has bounded geometry. The number $\delta$ is independent of $p$ and the bound on the geometry of $G_{p}$ is uniform as well.

We claim that $\Pi(\Sigma) = \mathbb{M}$. Suppose the contrary. Then there exists a bounded open set $\Omega \subset \Pi(\Sigma)$ and $q_0 \in \partial \Omega$ such that, for some point $p \in \Pi^{-1}(\Omega)$, a neighborhood of $p$ in $\Sigma$ is a vertical graph of a function $f$ defined over $\Omega$ and this graph does not extend to a minimal graph over any neighborhood of $q_0$. 

We can identify $\Omega$ with $\Omega\times\{0\}$. Let $\{q_n\} \subset \Omega$ be a sequence converging to $q_0$ and $p_{n} = \bigl(q_{n},f(q_{n})\bigr)$. Let $\Sigma_{n}$ denote the image of $G_{p_{n}}$ under the vertical translation taking $p_{n}$ to $q_{n}$. There is a subsequence of $\{q_n\}$ (which we also denote by $\{q_n\}$) such that the tangent planes $T_{q_n}(\Sigma_{n})$ converge to some vertical plane $P \subset T_{q_0} \bigl(\mathbb{M}\times\mathbb{R}\bigr)$. In fact, if this were not true, for $q_{n}$ close enough to $q_0$, the graph of bounded geometry $G_{p_n}$ would extend to a vertical graph beyond $q_0$. Hence $f$ would extend beyond $q_0$, a contradiction. So $T_{p_n}\Sigma$ must become almost vertical at $p_n$ for $n$ sufficiently large, which means that $\eta(p_n)$ must become horizontal. But $\nu$ is a constant different from 0, a contradiction. 

Then $\Pi(\Sigma) = \mathbb{M}$. Since $K_{\mathbb{M}}\circ\Pi \equiv 0$, it follows that $\mathbb{M}$ is a complete flat surface, which contradicts our assumption.
\end{proof}

\begin{lem}[\cite{H}]
Let $\Sigma$ be a minimal surface of $\mathbb{M}\times\mathbb{R}$. Then the height function $h: \mathbb{M}\times\mathbb{R} \rightarrow \mathbb{R}$, $h(q,t) = t$, is harmonic on $\Sigma$.
\end{lem}

\begin{proof}
Let $E_{1}, E_{2}, \eta$ be an orthonormal frame in a neighborhood of a point of $\Sigma$, where $\eta$ is normal to $\Sigma$. Since $\partial_{t}$ is a Killing vector field on $\mathbb{M}\times\mathbb{R}$, we have $$\diver\partial_{t} = 0 = \langle\nabla_{\eta}\partial_{t},\eta\rangle.$$

Write $\partial_{t} = \nabla h = X + \nabla_{\Sigma}h$, where $X$ is normal to $\Sigma$. Then
\begin{eqnarray*}
0 = \Delta h = \sum_{i}\bigl[\langle\nabla_{E_{i}}\nabla_{\Sigma}h,E_{i}\rangle + \langle\nabla_{E_{i}}X,E_{i}\rangle\bigr] \\ = \Delta_{\Sigma}h - \sum_{i}\langle X,\nabla_{E_{i}}E_{i}\rangle = \Delta_{\Sigma}h - \langle X,\vec{H}\rangle = \Delta_{\Sigma}h.
\end{eqnarray*} 
\end{proof}

\begin{lem}[\cite{H}]\label{lap}
Suppose that $\mathbb{M}$ has non-negative sectional curvature and that there exists a point $p \in \mathbb{M}$ such that the geodesic curvatures of all geodesic circles with center $p$ and radius $r \geq 1$ are uniformly bounded. Define $f: \mathbb{M}\backslash\bigl(\{p\}\cup \Cut(p)\bigr)\times\mathbb{R} \rightarrow \mathbb{R}$, $f(q,t) = \log (r(q))$, where $r$ is the distance in $\mathbb{M}$ to the point $p$. Let $\Sigma$ be a minimal surface of $\mathbb{M}\times\mathbb{R}$. Then, $$\Delta_{\Sigma}f \leq \frac{c_{1}}{r}|\nabla_{\Sigma}h|^{2},$$ for some constant $c_{1} > 0$ and $r \geq 1$.
\end{lem}

\begin{proof}
Denote by $\nabla f$, $\Delta f$ and $\hess{f}$ respectively the gradient, the laplacian and the hessian of $f$ in $\mathbb{M}\times\mathbb{R}$. Since $\mathbb{M}$ has non-negative curvature, by the Laplacian comparison theorem we have $$\Delta_{\mathbb{M}}r \leq \frac{1}{r}.$$ But $f$ does not depend on the height, so
$$\Delta f = \Delta_{\mathbb{M}}f = \frac{\Delta_{\mathbb{M}}r}{r} - \frac{|\nabla_{\mathbb{M}}r|^{2}}{r^{2}} \leq 0.$$ 

Let $E_{1}, E_{2}, \eta$ be an orthonormal frame in a neighborhood of a point of $\Sigma$, where $\eta$ is normal to $\Sigma$. Write $\nabla f = \nabla_{\Sigma}f + \langle\nabla f,\eta\rangle \eta$. Since $\Sigma$ is minimal we have
\begin{eqnarray*}
\Delta f &=& \sum_{i=1}^{2}\langle\nabla_{E_{i}}\nabla f,E_{i}\rangle + \langle\nabla_{\eta}\nabla f,\eta\rangle\\
&=& \sum_{i=1}^{2}\langle\nabla_{E_{i}}\nabla_{\Sigma}f,E_{i}\rangle + \sum_{i=1}^{2}\langle\nabla f,\eta\rangle\langle\nabla_{E_{i}}\eta,E_{i}\rangle + \langle\nabla_{\eta}\nabla f,\eta\rangle\\
&=& \Delta_{\Sigma}f + \langle\nabla f,\eta\rangle H + \hess{f}(\eta,\eta)\\\\
&=& \Delta_{\Sigma}f + \hess{f}(\eta,\eta).
\end{eqnarray*}

Now, let $V$ be tangent to $\mathbb{M}$, $\xi = \frac{\partial}{\partial t}$ and $\Pi$ be the projection of $\mathbb{M}\times\mathbb{R}$ to $\mathbb{M}$. Again, since $f$ does not depend on the height, we have 
\begin{eqnarray*}
\hess{f}(\xi,\xi) &=& 0,\\
\hess{f}(V,V) &=& \hess_{\mathbb{M}}{f}(V,V).
\end{eqnarray*}

Then, $$\hess{f}(\eta,\eta) = \hess{f}\bigl(\Pi(\eta),\Pi(\eta)\bigr) = \hess_{\mathbb{M}}{f}\bigl(\Pi(\eta),\Pi(\eta)\bigr).$$ 

But $\Delta f \leq 0$, so 
\begin{eqnarray}\label{hess}
\Delta_{\Sigma}f &\leq& -\hess{f}_{\mathbb{M}}\bigl(\Pi(\eta),\Pi(\eta)\bigr)\nonumber\\ 
&\leq& |\hess_{\mathbb{M}}{f}||\Pi(\eta)|^{2}.
\end{eqnarray}

A simple calculation shows that 
\begin{equation}\label{grad}
|\Pi(\eta)| = |\nabla_{\Sigma}h|.
\end{equation}
 
Let $q \in \mathbb{M}$, $r(q) = d(q,p)$ and $v$ be a unit tangent vector to $\mathbb{M}$ at $q$. Thus
\begin{eqnarray*}
\hess_{\mathbb{M}}{f}(v,v) &=& \biggl\langle\nabla_{v}\biggl(\frac{\nabla_{\mathbb{M}}r}{r}\biggr),v\biggr\rangle\\
&=& \frac{1}{r}\langle\nabla_{v}\nabla_{\mathbb{M}}r,v\rangle + v\biggl(\frac{1}{r}\biggr)\langle\nabla_{\mathbb{M}}r,v\rangle.
\end{eqnarray*}

When $v = \nabla_{\mathbb{M}}r$, $$\hess_{\mathbb{M}}{f}(v,v) = -\frac{1}{r^{2}}|\nabla_{\mathbb{M}}r|^{2}.$$

When $v = T$, the unit tangent vector to the geodesic circle of radius $r$ through the point $q$, $$\hess_{\mathbb{M}}{f}(v,v) = \frac{1}{r}\langle\nabla_{T}\nabla_{\mathbb{M}}r,T\rangle = \frac{1}{r}k_{g}(q),$$ where $k_{g}(q)$ is the geodesic curvature of the geodesic circle of radius $r$ centered at the point $q$. By the hypothesis about the geodesic circles of $\mathbb{M}$,
$$|\hess_{\mathbb{M}}{f}|^{2} = \frac{1}{r^{4}} + \frac{1}{r^{2}}k_{g}^{2} \leq  \frac{C}{r^{2}}.$$

Using equations \eqref{hess} and \eqref{grad}, the lemma follows.
\end{proof}

\section{Laminations}

\begin{defin}
Let $\Sigma$ be a complete, embedded surface in a $3$-manifold $N$. A point $p \in N$ is a limit point of $\Sigma$ if there exists a sequence $\{p_{n}\} \subset \Sigma$ which diverges to infinity in $\Sigma$ with respect to the intrinsic Riemannian topology on $\Sigma$, but converges in $N$ to $p$ as $n \to \infty$. Let $\mathcal{L}(\Sigma)$ denote the set of all limit points of $\Sigma$ in $N$; we call this set the limit set of $\Sigma$. In particular, $\mathcal{L}(\Sigma)$ is a closed subset of $N$ and $\bar\Sigma\backslash\Sigma \subset \mathcal{L}(\Sigma)$, where $\bar\Sigma$ denotes the closure of $\Sigma$.
\end{defin}

\begin{defin}
A codimension-$1$ lamination of a Riemannian $n$-manifold $N$ is the union of a collection of pairwise disjoint, connected, injectively immersed hypersurfaces, with a certain local product structure. More precisely, it is a pair $(\mathcal{L}, \mathcal{A})$ satisfying the following conditions:
\begin{enumerate}
\item $\mathcal{L}$ is a closed subset of $N$;
\item $\mathcal{A} = \{\varphi_{\beta}: \mathbb{D}\times (0,1) \rightarrow U_{\beta}\}_{\beta}$ is an atlas of coordinate charts of $N$, where $\mathbb{D}$ is the open unit ball in $\mathbb{R}^{n-1}$ and $U_{\beta}$ is an open subset of $N$;
\item For each $\beta$, there is a closed subset $C_{\beta}$ of $(0,1)$ such that $\varphi_{\beta}^{-1}(U_{\beta}\cap\mathcal{L}) =  \mathbb{D}\times C_{\beta}$.
\end{enumerate}
If all the leaves are minimal hypersurfaces, $(\mathcal{L}, \mathcal{A})$ is called a minimal lamination.
\end{defin}

\section{Parabolic Manifolds}

\begin{defin}
Given a point $p$ on a Riemannian manifold $N$ with boundary, one can define the hitting, or harmonic measure $\mu_{p}$ of an interval $I \subset \partial N$, as the probability that a Brownian path beginning at $p$ reaches the boundary for the first time at a point in $I$.
\end{defin}

\begin{prop}
Let $N$ be a Riemannian manifold with non empty boundary. The following are equivalent:
\begin{enumerate}
\item Any bounded harmonic function on $N$ is determined by its boundary values;
\item For some $p \in Int(N)$, the measure $\mu_{p}$ is full on $\partial N$, i.e, $\int_{\partial N}\mu_{p} = 1$;
\item If $h: N \rightarrow \mathbb{R}$ is a bounded harmonic function, then $h(p) = \int_{\partial N}h(x)\mu_{p}$;
\end{enumerate}
If $N$ satisfies any of the conditions above, then it is called parabolic.
\end{prop}

An important property is that a proper subdomain of a parabolic manifold is parabolic, hence removing the interior of a compact domain does not alter parabolicity. Moreover if there exists a proper non negative superharmonic function on $N$, then $N$ is parabolic. For equivalent definitions and properties of parabolic manifolds see \cite{G}.

\begin{defin}
Let $N$ be a Riemannian manifold with non empty boundary. For $R > 0$, let $N(R) = \{p \in N; d(p,\partial N)< R\}$. We say that $N$ is $\delta$-parabolic if for every $\delta > 0$, $\widetilde{N} = N\backslash N(\delta)$ is parabolic.
\end{defin}

The following theorem gives a sufficient condition for a surface to be $\delta$-parabolic, for the proof see \cite{M.R1}.

\begin{teo1}\label{parab}
Let $N$ be a complete surface with non empty boundary and curvature function $K: N \rightarrow [0,\infty]$. Suppose that for each $R > 0$, the restricted function $K|_{N(R)}$ is bounded. Then $N$ is $\delta$-parabolic.
\end{teo1}

\section{Proper Minimal Immersions}

\begin{prop}\label{deltapar}
Let $N$ be a $3$-manifold with non negative Ricci curvature and sectional curvature bounded above by $\kappa > 0$. Suppose $\Sigma$ is a complete, orientable minimal surface with boundary in $N$, with a Jacobi function $u$. If $u \geq \epsilon$, for some $\epsilon > 0$, then $\Sigma$ is $\delta$-parabolic.
\end{prop}

\begin{proof}
First note that a Riemannian surface $W$ is $\delta$-parabolic if and only if for all $\delta' > 0$, the surface $W\backslash W(\delta')$ is also $\delta$-parabolic. Thus, without loss of generality, we may assume that $\Sigma$ has the form $W\backslash W(\delta')$ for some $\delta' > 0$, where $W$ is a stable minimal surface with a positive Jacobi function $u \geq \epsilon$, which exists by \cite{FC.S}. By curvature estimates for stable, orientable minimal surfaces (\cite{Sc, R.S.T}), we may assume that $\Sigma$ has bounded Gaussian curvature. Consider the new Riemannian manifold $\widetilde{\Sigma}$, which is $\Sigma$ with the metric $\tilde{g} = u\langle\cdot,\cdot\rangle$ on $\Sigma$, where $\langle\cdot,\cdot\rangle$ is the Riemannian metric on $\Sigma$. Since $u \geq \epsilon$ the metric $\tilde{g}$ is complete. Moreover $\Delta_{\tilde{g}}f = u^{-1}\Delta f$, for any function on $\Sigma$ which has second derivative. Thus $\Sigma$ is $\delta$-parabolic if and only if $\widetilde{\Sigma}$ is $\delta$-parabolic.
Let $E_{1}, E_{2}, \eta$ be an orthonormal frame in a neighborhood of a point of $\Sigma$, where $\eta$ is normal to $\Sigma$. By the Gauss equation 
$$\ric(\eta,\eta) + |A_{\Sigma}|^{2} = \ric(E_{1},E_{1}) + \ric(E_{2},E_{2}) - 2K_{\Sigma}.$$ 
Then, as $u$ is a Jacobi function 
$$\Delta_{\Sigma}u + \bigl(\ric(E_{1},E_{1}) + \ric(E_{2},E_{2}) - 2K_{\Sigma}\bigr)u = 0.$$ 

So,
$$K_{\widetilde{\Sigma}} = \frac{K_{\Sigma} - \frac{1}{2}\Delta_{\Sigma} \log u}{u} = \frac{1}{2}\frac{\ric(E_{1},E_{1}) + \ric(E_{2},E_{2})}{u} + \frac{1}{2}\frac{|\nabla_{\Sigma}u|^{2}}{u^{3}},$$ 
which implies 
$$0 \leq K_{\widetilde{\Sigma}} \leq 2\frac{\kappa}{\epsilon} + \frac{1}{2\epsilon}\frac{|\nabla_{\Sigma}u|^{2}}{u^{2}}.$$

Choose $\delta > 0$ and let $\widetilde\Omega = \widetilde{\Sigma}\backslash \widetilde{\Sigma}(\delta)$. Let $\Omega$ be the corresponding submanifold on $\Sigma$. By the Harnack inequality, see \cite{Mo}, $\frac{|\nabla_{\Sigma}u|}{u}$ is bounded, and so one has that $K_{\widetilde{\Sigma}}$ is non negative and bounded on $\Omega$. It follows from theorem $1$ in section 4 that $\widetilde\Omega$ is parabolic, and hence $\Omega$ is parabolic. Since $\delta$ was chosen arbitrarily, we conclude that $\Sigma$ is $\delta$-parabolic.
\end{proof}

\begin{teoA}
Let $\mathbb{M}$ be a complete simply connected orientable non-compact surface such that $0 \leq K_{\mathbb{M}} \leq \kappa$. Let $f: \Sigma \rightarrow \mathbb{M}\times\mathbb{R}$ be an injective minimal immersion of a complete, connected Riemannian surface of bounded curvature. Then the map $f$ is proper.
\end{teoA}

\begin{proof}
Since the curvature of $f(\Sigma)$ is bounded, there exists an $\epsilon > 0$ such that for any point $p \in \mathbb{M}\times\mathbb{R}$, every component of $f^{-1}\bigl(B_{\epsilon}(p)\bigr)$, when pushed forward by $f$, is a compact disc and a graph over a domain in the tangent plane of any point on it, with a uniform bound on the area. It follows that if $p$ is a limit point of $f(\Sigma)$ coming from distinct components of $f^{-1}\bigl(B_{\epsilon}(p)\bigr)$, then there is a minimal disc $D(p)$ passing through $p$ that is a graph over its tangent plane at $p$, and $D(p)$ is a limit of components in $f^{-1}\bigl(B_{\epsilon}(p)\bigr)$. Let $D'(p)$ be any other such limit disc. Since $f$ is an embedding the unique possibility is that the discs are tangent at $p$, then the maximum principle imply that the two discs agree near $p$. This implies that the closure $\mathcal{L}\bigl(f(\Sigma)\bigr)$ of $f(\Sigma)$ has the structure of a minimal lamination.

The immersion $f$ is proper if and only if $\mathcal{L}\bigl(f(\Sigma)\bigr)$ has no limit leaves. Suppose $\mathcal{L}\bigl(f(\Sigma)\bigr)$ has a limit leaf $L$. Denote by $\widetilde{L}$ the universal cover of $L$. It was proved in \cite{M.P.R} that $\widetilde{L}$ is stable. So, by \cite{FC.S} $\widetilde{L}$ is totally geodesic, hence $L$ is totally geodesic. Suppose $\mathbb{M}$ is not flat (the case where $\mathbb{M}$ is flat was proved in \cite{M.R1}). By lemma \ref{tot.geod} a totally geodesic surface in $\mathbb{M}\times\mathbb{R}$ is a slice $\mathbb{M}\times\{t\}$, or is of the form $\alpha\times\mathbb{R}$, where $\alpha$ is a geodesic of $M$. 

Assume $L$ is a slice. Since $\Sigma$ is not proper, it is not equal to a slice. We can suppose $L = \mathbb{M}\times\{0\}$ and $H^{+}$ is a smallest halfspace containing $f(\Sigma)$. Since $\Sigma$ has bounded curvature, there is an $\epsilon > 0$ such that for every component $C$ of $f(\Sigma)$ in the slab between $L$ and $L_{\epsilon} = \{t = \epsilon\}$, the Jacobi function $u = \langle \nu,\partial_{t}\rangle$ satisfies $u \geq \lambda > 0$, where $\nu$ is the unit normal to $C$. Choose $0 < \delta < \epsilon$ such that $C(\delta) = \{p \in C; h \leq \delta\}$ is not empty, where $h$ is the height function. By proposition \ref{deltapar}, $C(\delta)$ is parabolic. But $h|_{C(\delta)}$ is a bounded harmonic function with the same boundary values as the constant function $\delta$. Hence $h|_{C(\delta)}$ is constant, which is a contradiction because $C(\delta)$ is not contained in a slice.

Now, suppose $L = \alpha\times\mathbb{R}$. Consider a one-sided closed $\epsilon$-normal interval bundle $N_{\epsilon}(L)$ that submerses to $\mathbb{M}\times\mathbb{R}$, with the induced metric. Observe that $N_{\epsilon}(L)$ is diffeomorphic to $\bigl(\alpha\times\mathbb{R}\bigr)\times [0,\delta]$, with $L = \bigl(\alpha\times\mathbb{R}\bigr)\times\{0\}$ as a flat minimal submanifold, and $L(\delta) = \bigl(\alpha\times\mathbb{R}\bigr)\times\{\delta\}$ having mean curvature vector pointing out of $N_{\epsilon}(L)$. For $\epsilon$ sufficiently small, we may assume that each component of $f(\Sigma)\cap N_{\epsilon}(L)$ is a normal graph of bounded gradient over the zero section $L$. Let $C$ be such a component which is a graph over a connected domain $\Omega$ of $L$ and let $L_{C}(\delta)$ be the part of $L_{\delta}$ which is also a normal graph over $\Omega$. Consider the surface $W_{\delta} = L(\delta)\backslash L_{C}(\delta)$. Under normal projection to $L$, $W_{\delta}\cup C$ is quasi-isometric to the flat plane $L$. It follows that $C$ is a parabolic Riemann surface with boundary. But the function $d := \dist(\cdot,L)$ is superharmonic, and has constant value $\delta$ on the boundary of $C$. Then $C$ is contained in $L(\delta)$, which contradicts the fact that $L$ is a limit leaf of $\mathcal{L}\bigl(f(\Sigma)\bigr)$.

\end{proof}

\section{Parabolicity of Minimal Surfaces}

\begin{teoB}
Let $\mathbb{M}$ be a complete non-compact surface satisfying the following conditions:
\begin{enumerate}
\item $K_{\mathbb{M}} \geq 0$; 
\item There is a point $p \in M$ such that the geodesic curvatures of all geodesic circles with center $p$ and radius $r \geq 1$ are uniformly bounded.
\end{enumerate}
Let $\Sigma$ be a surface of finite topology and one end and let $f: \Sigma \rightarrow \mathbb{M}\times\mathbb{R}$ be a proper minimal immersion. Suppose that $f$ is transverse to a slice $\mathbb{M}\times\{t_0\}$ except at a finite number of points, and that $f^{-1}(\mathbb{M}\times\{t_0\})$ contains a finite number of components. Then $\Sigma$ is parabolic.
\end{teoB}

\begin{proof}
We know from \cite{H} that the conditions on $\mathbb{M}$ imply that the surfaces
\begin{eqnarray*}
\Sigma(+) &:=& \{(p,t) \in \Sigma; t \geq t_0\},\\
\Sigma(-) &:=& \{(p,t) \in \Sigma; t \leq t_0\}
\end{eqnarray*} 
are parabolic. Suppose that $\mathcal{E}$ is an annular end representative which does not have conformal representative which is a punctured disc. Then this end has a representative which is conformally diffeomorphic to $\{z \in \mathbb{C}; \epsilon \leq |z| < 1\}$ for some positive $\epsilon < 1$. In this conformal parametrization, the unit circle corresponds to points at infinity on $\mathcal{E}$. After choosing a larger $\epsilon$, we may assume that $f|_{\mathcal{E}}$ intersects $\mathbb{M}\times\{t_0\}$ transversely in a finite positive number of arcs and that each noncompact arc of the intersection has one endpoint on the compact boundary circle $\{z \in \mathbb{C}; |z| = \epsilon\}$.

\begin{figure}[!ht]
\center
\includegraphics[scale=0.5]{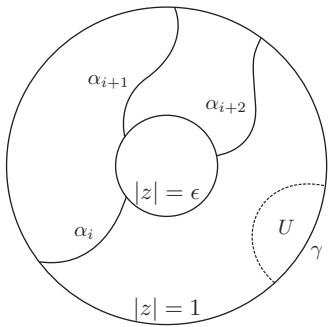}
\label{fig1}
\caption{The disc $U$}
\end{figure} 

We claim that it suffices to prove that each of the finite number of noncompact arcs $\alpha_{1},\dots,\alpha_{n}$ in $\mathbb{M}\times\{t_0\}$ has a well-defined limit on the unit circle $\mathbb{S}^{1}$ of points at infinity. In fact, assume the claim is true, then there is an open arc $\gamma \subset \mathbb{S}^{1}$ which does not contain limit points of $\alpha_{1},\dots,\alpha_{n}$. Hence, there would be an open half-disc $U \subset \mathcal{E}$ centered at a point in $\gamma$, such that $U\cap\bigl(f^{-1}(\mathbb{M}\times\{t_0\})\bigr) = \emptyset$, see figure $1$. But $U$ is a proper domain which is contained in one of the parabolic surfaces $\Sigma(+)$ or $\Sigma(-)$, so is parabolic. However $U$ does not have full harmonic measure, which is a contradiction.

Suppose $\alpha_{k}$ has two limit points $q_{1},q_{2}$ in $\mathbb{S}^{1}$. We first prove that at least one of the two interval components $I_1$, $I_2$ of $\mathbb{S}^{1}\backslash\{q_{1},q_{2}\}$ consists of limit points of $\alpha_{k}$. Suppose not and let $x_{1} \in I_{1}, x_{2} \in I_{2}$ be points which are not limit points. Since they are not limit points, there exists a $\delta > 0$ such that the radial arcs $\beta_{1}$ and $\beta_{2}$ in $\mathcal{E}$ of length $\delta$ and orthogonal to $\mathbb{S}^{1}$ at $x_{1}, x_{2}$ respectively, are disjoint from $\alpha_{k}$. Since $\alpha_{k}$ is proper and disjoint from $\beta_{1}\cup\beta_{2}$, the parametrized arc $\alpha_{k}(s)$ must eventually be in one of the two components of $\{z \in \mathcal{E}\backslash(\beta_{1}\cup\beta_{2});|z| \geq 1 - \delta\}$; see figure $2$. 
Thus, $\alpha_{k}$ cannot have both $q_{1}$ and $q_{2}$ as limit points, a contradiction. Now, suppose one of the intervals, say $I_{2}$, contains one point $z$ which is not a limit point of $\alpha_{k}$, then by the previous argument the interval $I_{1}$ cannot contain any point which is not a limit point. So one of the intervals consists of limit points of $\alpha_{k}$.

\begin{figure}[!ht]
\center
\includegraphics[scale=0.5]{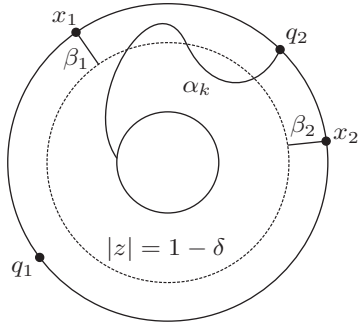}
\label{fig2}
\caption{The arc trapped in one of the components.}
\end{figure}

Since the height function is harmonic on $\mathcal{E}$ and the generator of the homology of $\mathcal{E}$ is a boundary in $\Sigma$, by Cauchy's Theorem there is a conjugate harmonic function to $h$, which we denote by $h^{*}$. Consider the holomorphic function $g = h + ih^{*}: \mathcal{E} \rightarrow \mathbb{C}$. As the slice $\mathbb{M}\times\{t_0\}$ is transverse to $\mathcal{E}$, we have $\langle\nabla h,\eta\rangle^{2} \neq 1$ for all points in an arc $\alpha_{k}$ and $h = 0$ in this arc, where $\eta$ is a unit normal to $\Sigma$. Moreover as $g$ is holomorphic we have $$|\nabla_{\Sigma} h^{*}(p)|^{2} = |\nabla_{\Sigma} h(p)|^{2} = 1 - \langle\nabla h,\eta\rangle^{2}(p) > 0, \forall p \in \alpha_{k},$$ so 
$h^{*}|_{\alpha_{k}}$ is strictly monotone. Thus $g$ restricted to any of the finite number of components in $\bigl(f^{-1}(\mathbb{M}\times\{t_0\})\bigr)\cap\mathcal{E}$, monotonically parametrizes an interval on the imaginary axis $\mathbb{R}(i) \subset \mathbb{C}$. Choose a closed half disc $\overline{D} \subset \overline{\mathcal{E}} = \mathcal{E}\cup\mathbb{S}^{1}$, centered at a point $p \in I_{1}$, where $I_{1}$, as discussed above, consists entirely of limit points of $\alpha_{1}$, and suppose that $\overline{D}$ is chosen sufficiently small so that $\partial_{\infty}D := \partial D\cap\mathbb{S}^{1} \subset I_{1}$. Since $g|_{\alpha_{k}}$ is injective we can take a compact interval $J \subset \displaystyle g(\cup_{k = 1}^{n}\alpha_{k}) \subset \mathbb{R}(i)$ which is disjoint from the endpoints of $g|_{\alpha_{k}}$, for all $k$, and choose $D$ sufficiently small such that $\overline{D}\cap\bigl(g^{-1}(J)\bigr) = \emptyset$. 

Observe that $g$ maps $D$ into $\mathbb{C}\backslash J$, so by the Riemann mapping theorem, the function $g|_{D}$ is essentially bounded in the sense that it maps $D$ into a domain that is conformally equivalent to an open subset of the unit disc.  It follows from Fatou's theorem that the holomorphic function $g|_{D}$ has radial limits almost everywhere, i.e., $D$ is conformally the unit disc, so radial limits are with respect to the radii of the unit disc.

\begin{figure}[!ht]
\centering
\includegraphics[scale=0.8]{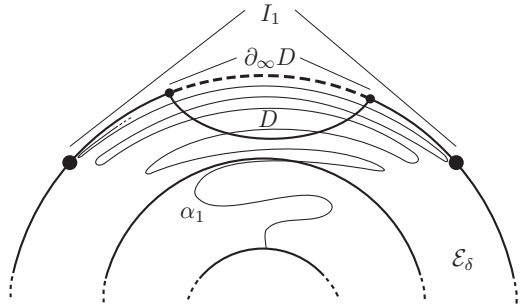}
\label{fig3}
\caption{The arc $\alpha_1$ accumulates in $I_{1}$}
\end{figure}

Consider the radial arc $\beta$ orthogonal to $\mathbb{S}^{1}$ at the point $p$ (the center of $I_{1}$). The arc $\beta$ divides $I_{1}$ into two intervals $I_{1}^{-}$ and $I_{1}^{+}$ and separates $D$ into two regions $D^{-}$ and $D^{+}$. Choose $\delta > 0$ small. We can suppose $D$ is inside the region $\mathcal{E}_{\delta} := \{z \in \mathcal{E};|z| \geq 1 - \delta\}$. Since $\alpha_{1}$ is proper, this arc will eventually be inside of $\mathcal{E}_{\delta}$. As $I_{1}$ is composed of accumulation points of $\alpha_{1}$ and $\partial_{\infty}D$ is not equal to $I_{1}$, the arc $\alpha_{1}$ leaves $D$ and returns to it an infinite number of times, and it does this crossing the boundaries of $D^{-}$ and $D^{+}$ infinitely many times, in each step getting closer to $\partial_{\infty}D^{-}$ and $\partial_{\infty}D^{+}$ respectively, see figure $3$. Then there exists an infinite number of arcs in $\alpha_{1}\cap D^{-}$ (respectively $\alpha_{1}\cap D^{+}$) converging to $\partial_{\infty}D^{-}$ (respectively $\partial_{\infty}D^{+}$), see figure $4$. Thus the points of $\partial_{\infty}D$ with radial limits for $g$ have a constant value which corresponds to the limiting endpoint of the curve  $g\circ\alpha_{1}$ in $\mathbb{R}(i)\cup\{\infty\}$. However by Privalov's theorem, a nonconstant meromorphic function on the unit disc cannot have a constant radial limit on a set of $\partial_{\infty}D$ with positive measure, a contradiction.

\begin{figure}[!ht]
\centering
\includegraphics[scale=0.8]{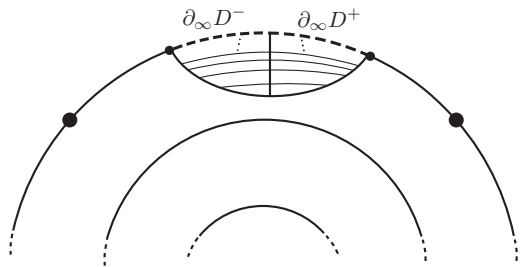}
\label{fig4}
\caption{Infinitely many arcs in $D^{-}$ and $D^{+}$}
\end{figure}

\end{proof}

\begin{teoC}
Let $\mathbb{M}$ be a complete non-compact surface satisfying the following conditions:
\begin{enumerate}
\item $0 \leq K_{\mathbb{M}} \leq \kappa$;
\item $\mathbb{M}$ has a pole $p$; 
\item The geodesic curvatures of all geodesic circles with center $p$ and radius $r \geq 1$ are uniformly bounded.
\end{enumerate} 
Let $\Sigma$ be a properly immersed minimal surface inside the region of $\mathbb{M}\times\mathbb{R}$ defined by $|h|\leq c_{2}\log r$, for some constant $c_{2}>0$ and $r\geq1$. Then $\Sigma$ is parabolic. Moreover, if $\Sigma$ has compact boundary, then $\Sigma$ has quadratic area growth and a finite number of ends. 
\end{teoC}

\begin{proof}
Let $p$ be the pole of $\mathbb{M}$. Since the map $\exp_{p}: T_{p}\mathbb{M} \rightarrow \mathbb{M}$ is a diffeomorphism, we have that $\phi: T_{p}\mathbb{M}\times\mathbb{R} \rightarrow \mathbb{M}\times\mathbb{R}$, defined by $\phi(v,s) = (\exp_{p}v,s)$ is a diffeomorphism and defines a coordinate system.

Let $r$ be the distance to $p$ on $\mathbb{M}$ extended to $\mathbb{M}\times\mathbb{R}$ in the natural way and $h$ be the height function on $\mathbb{M}\times\mathbb{R}$. Let $C_{R} = \{(q,s) \in \mathbb{M}\times\mathbb{R}; r(q) = R\}$ be the vertical cylinder of radius $R$ and let $\Sigma_{R}$ be the part of $\Sigma$ inside $C_{R}$. Let $B_{R}\bigl((p,0)\bigr)$ be the ball of $\mathbb{M}\times\mathbb{R}$ of center $(p,0)$ and radius $R$. Since $\mathbb{M}\times\mathbb{R}$ has the product metric and $p$ is a pole in $M$, the point $(p,0)$ is a pole in $\mathbb{M}\times\mathbb{R}$. Thus $\Sigma\cap B_{R}\bigl((p,0)\bigr)$ is inside the interior of $C_{R}$. Then it suffices to prove that $\Sigma_{R}$ has quadratic area growth as a function of $r$.

Using these coordinates we can define a horizontal vector field $X$ that is orthogonal to $\nabla r$ and $\nabla h$ and has norm 1, so $(\nabla r, \nabla h, X)$ is an orthonormal basis at each point of $\mathbb{M}\times\mathbb{R}$. Let $\eta$ be a unit normal to $\Sigma$, so
$$\langle \eta,\nabla r\rangle^{2} + \langle \eta,\nabla h\rangle^{2} + \langle \eta,X\rangle^{2} = 1,$$ 
$$|\nabla_{\Sigma}r|^{2} = 1 - \langle \eta,\nabla r\rangle^{2},$$ and
$$|\nabla_{\Sigma}h|^{2} = 1 - \langle \eta,\nabla h\rangle^{2}.$$
Hence, 
$$|\nabla_{\Sigma}r|^{2} + |\nabla_{\Sigma}h|^{2} = 1 + \langle \eta,X\rangle^{2} \geq 1.$$
Thus,
$$\int_{\Sigma_{R}}d\mu \leq \int_{\Sigma_{R}}\bigl(|\nabla_{\Sigma}r|^{2} + |\nabla_{\Sigma}h|^{2}\bigr)d\mu.$$

Consider the function $f: \Sigma \rightarrow \mathbb{R}$, $f = - h\arctan (h) + \frac{1}{2}\log(h^{2} + 1),$ where $h$ is the height function on $\mathbb{M}\times\mathbb{R}$. Since $h$ is harmonic on $\Sigma$, 
\begin{eqnarray*}
\Delta_{\Sigma}f &=& - \arctan(h)\Delta_{\Sigma}h -\frac{|\nabla_{\Sigma}h|^{2}}{h^2 + 1}\\
&=& -\frac{|\nabla_{\Sigma}h|^{2}}{h^2 + 1}.
\end{eqnarray*}

Consider now the function $g = \log r + f$. After rescaling the metric of $\Sigma$ and removing a compact subset of $\Sigma$ we may assume that $|h| \leq \frac{1}{2}\log r$. By lemma \ref{lap}, $g$ satisfies 
$$\Delta_{\Sigma}g \leq c_{1}\frac{|\nabla_{\Sigma}h|^{2}}{r} -\frac{|\nabla_{\Sigma}h|^{2}}{h^2 + 1} \leq 0.$$ 

Since $\log r$ is proper in $\{(q,t) \in \mathbb{M}\times\mathbb{R} \ | \ |h| \leq \frac{1}{2}\log r, r \geq 1\}$ and $\Sigma$ is proper, $\log r$ is proper in $\Sigma$. Moreover $g \geq \frac{3\pi}{4}\log r$, so $g$ is a non-negative proper superharmonic function on $\Sigma$. This proves that $\Sigma$ is parabolic.\\

Suppose $\partial\Sigma$ is compact. There exists $a > 0$ such that $g(\partial\Sigma) \subset [0,a]$. Let $t_{2} > t_{1} \geq a$. Since $g$ is proper, $g^{-1}\bigl([t_{1},t_{2}]\bigr)$ is compact; then we can apply the divergence theorem and use the fact that $g$ is superharmonic to obtain
\begin{equation}\label{mon.seq}
0 \geq \int_{g^{-1}([t_{1},t_{2}])}\Delta_{\Sigma}g\ d\mu = -\int_{g^{-1}(t_{1})}|\nabla_{\Sigma}g|\ dL + \int_{g^{-1}(t_{2})}|\nabla_{\Sigma}g|\ dL.
\end{equation}

It follows that the function $t \mapsto \int_{g^{-1}(t)}|\nabla_{\Sigma}g|\ dL$ is monotonically decreasing and bounded, so 
\begin{equation}\label{int.func}
\lim_{t \to \infty}\int_{g^{-1}(t)}|\nabla_{\Sigma}g|\ dL < \infty.
\end{equation}

Since $\Sigma = g^{-1}\bigl([0,\infty)\bigr)$ it follows from \eqref{mon.seq} and \eqref{int.func} that $\Delta_{\Sigma}g \in \lp^{1}(\Sigma)$. Furthermore, $\Delta_{\Sigma}g \geq \frac{1}{2}|\Delta_{\Sigma}f|$, for $r$ large, thus $\Delta_{\Sigma}f \in \lp^{1}(\Sigma)$. Hence,
$$\int_{\Sigma_{R}}\Delta_{\Sigma}f\ d\mu = \int_{\Sigma_{R}}\frac{|\nabla_{\Sigma}h|^{2}}{h^2 + 1}d\mu \leq \int_{\Sigma}\frac{|\nabla_{\Sigma}h|^{2}}{h^2 + 1}d\mu = c_{3},$$ for some positive constant $c_{3}$. Then, for $R \geq 1$,
\begin{eqnarray*}
\int_{\Sigma_{R}}|\nabla_{\Sigma}h|^{2}d\mu &\leq & \int_{\Sigma_{R}}\biggl(\frac{(\log R)^{2} + 1}{h^2 + 1}\biggr)|\nabla_{\Sigma}h|^{2}d\mu\\\\ 
&\leq & \bigl((\log R)^{2} + 1\bigr)c_{3} \leq c_{3}R^{2}. 
\end{eqnarray*}

Since $\Delta_{\Sigma}f \in \lp^{1}(\Sigma)$ and $|\Delta_{\Sigma}f| \geq c_{4}|\Delta_{\Sigma}\log r|$ ($c_{4} > 0$ a constant), we have $\Delta_{\Sigma}(\log r) \in \lp^{1}(\Sigma)$. Again by the Divergence Theorem,
\begin{eqnarray*}
\int_{\Sigma_{R}}\Delta_{\Sigma}\log r\ d\mu &=& \int_{\partial\Sigma}\frac{1}{r}\langle\nabla_{\Sigma}r,\nu\rangle dL + \int_{C_{R}\cap\Sigma}\frac{|\nabla_{\Sigma}r|}{R}dL\\
&=& {c_5} + \frac{1}{R}\int_{C_{R}\cap\Sigma}|\nabla_{\Sigma}r|dL,
\end{eqnarray*}
where $\nu$ is the outward conormal to the boundary of $\Sigma$. Thus $$\lim_{R \to \infty}\frac{1}{R}\int_{C_{R}\cap\Sigma}|\nabla_{\Sigma}r|dL < \infty,$$ which implies there is a constant $c_{6} > 0$ such that
$$\int_{C_{R}\cap\Sigma}|\nabla_{\Sigma}r|dL \leq c_{6}R.$$ 

By the coarea formula
$$\int_{\Sigma_{R}}|\nabla_{\Sigma}r|^{2}d\mu \leq \int_{1}^{R}\int_{C_{\rho}\cap\Sigma}|\nabla_{\Sigma}r|dLd\rho \leq c_{6}\int_{1}^{R}\rho \ d\rho \leq \frac{1}{2}c_{6}R^{2}.$$

Therefore $\Sigma$ has quadratic area growth.\\

Now, suppose $\Sigma$ has an infinite number of ends. Let $E$ be an end of $\Sigma$. Choose $0 < \delta < \min\bigl\{\inj_{\mathbb{M}\times\mathbb{R}},\frac{1}{\sqrt{\kappa}}\bigr\}$ such that for each positive integer $j$, there is a distance ball $B_{\delta}(q_{j})$ of $\mathbb{M}\times\mathbb{R}$ inside the region $\mathcal{R}_{j}$ between ${C_{j}}$ and $C_{j+1}$, with $q_{j} \in E$. By the monotonicity formula for minimal surfaces (see chapter 7 of \cite{C.M3}) $$|E\cap B_{\delta}(q_{j})| \geq \frac{c\delta^2}{e^{2\sqrt{\kappa}\delta}} =: c_{7},$$
where $c > 0$ is a constant and $\kappa = \sup K_{\mathbb{M}\times\mathbb{R}}$. Write $E_{n} = E\cap C_{n}$. Since in each region $\mathcal{R}_{j}$, $j < n$, we have a portion of $E$ of area at least $c_{7}$ it follows that $$|E_{n}| > c_{7} n.$$

Then in the cylinder $C_{n^2}$ we have $$c_{7} n^{2} \leq |E_{n^2}| \leq c_{8}n^{2}.$$ 

Since this holds for each end, choosing $n$ ends we obtain that the area of $\Sigma$ inside $C_{n^2}$ satisfies $$c_{9} n^{3} \leq |\Sigma_{n^2}| \leq c_{10}n^{2},$$ but for $n$ sufficiently large this leads to a contradiction. Hence, $\Sigma$ has a finite number of ends.

\end{proof}

\noindent \textsc{Instituto de Matem\'atica e Estat\'istica, UERJ\\ Rua S\~ao Francisco Xavier, 524\\Pavilh\~ao Reitor Jo\~ao Lyra Filho, 6º andar - Bloco B\\ 20550-900, Rio de Janeiro-RJ, Brazil}
\\


\begin{thebibliography}{20}

\bibitem{B.J.O}  G.P. Bessa, L.P. Jorge, G .Oliveira-Filho: {\it Half-space theorems for minimal surfaces with bounded curvature}. J. Differential Geom. 57 (2001), no. 3, 493-508.

\bibitem{C.H.M} M. Callahan, D. Hoffman \& W.H. Meeks III, {\it The structure of singly- periodic minimal surfaces}. Invent. Math. 99 (1990) 455–481, MR 1032877, Zbl 0695.53005.

\bibitem{C.M2} T.H. Colding, W.P. Minicozzi, II: {\it The Calabi-Yau conjectures for embedded
surfaces}. Ann. of Math. (2) 167 (2008), no. 1, 211-243.

\bibitem{C.M3} T. H. Colding \& W. P. Minicozzi II: {\it A Course in Minimal Surfaces}. AMS (2011), ISBN-10:  0-8218-5323-6, ISBN-13: 978-0-8218-5323-8.

\bibitem{C} P. Collin: {\it Topologie et courbure des surfaces minimales proprement plong\'ees de $\mathbb{R}^{3}$}. Annals of Mathematics (145)(1997), p.1-31.

\bibitem{C.K.M.R} P. Collin, R Kusner, W.H. Meeks III \& H. Rosenberg, {\it The topology, geometry and conformal structure of properly embedded minimal surfaces}. Journal of Differential Geometry, v. 67, p. 377-393, 2004.
 
\bibitem{MC} M. P. do Carmo, Geometria Riemanniana, Rio de Janeiro, Projeto Euclides, IMPA, 2a edição, 1988.

\bibitem{E.R} J.M. Espinar, H.Rosenberg. Complete constant mean curvature surfaces and Bernstein type theorems in $M^{2}\times\mathbb{R}$. J. Differential Geom. 82 (2009), no. 3, 611-628.

\bibitem{FC.S} D. Fischer-Colbrie \& R. Schoen, {\it The structure of complete stable minimal surfaces in 3-manifolds of nonnegative scalar curvature}. Comm. Pure Appl. Math. 33 (1980), no 2, 199-211.

\bibitem{F.M} C. Frohman \& W.H. Meeks III, {\it The ordering theorem for the ends of properly embedded minimal surfaces}. Topology 36(3) (1997) 605–617, MR 1422427, Zbl 0878.53008.

\bibitem{G} A. Grigor’yan, {\it Analytic and geometric background of recurrence and non-explosion of the Brownian motion on Riemannian manifolds}. Bull. Amer. Math. Soc. (N.S.) 36 (1999), no. 2, 135 – 249. 

\bibitem{H.M} D. Hoffman and William H. Meeks, III. {\it The strong halfspace theorem
for minimal surfaces}. Invent. Math., 101:373-377, 1990.

\bibitem{L} G.Liu, {\it Three-manifolds with nonnegative Ricci Curvature}. Inventiones mathematicae August 2013, Volume 193, Issue 2, pp 367-375.

\bibitem{M.P.R} W. H. Meeks, III; J.P\' erez, A. Ros: {\it Stable Constant Mean Curvature Surfaces}. Handbook of geometric analysis. No. 1, 301-380.]

\bibitem{M.R} W.H. Meeks, III; H. Rosenberg: {\it The uniqueness of the helicoid}. Annals of Mathematics, (161) (2005), 727-758.

\bibitem{M.R1} W.H. Meeks III \& H. Rosenberg, {\it Maximum Principles at Infinity}. Journal of Differential Geometry, v. 79, no. 1, p. 141-165, 2008.

\bibitem{M.R2} W.H. Meeks III \& H. Rosenberg, {\it The Minimal Lamination Closure Theorem}. Duke Journal of Math, v. 133, no. 3, p. 467-497, 2006.

\bibitem{Mo} J. Moser, {\it On Harnack's theorem for elliptic differential equations}. Comm. Pure Appl. Math. 14 1961 577–591.

\bibitem{H} H. Rosenberg, {\it Minimal surfaces in $M\times\mathbb{R}$}. Illinois J. Math. 46 (2002), no. 4, 1177–1195.

\bibitem{R.S.T} H. Rosenberg, R. Souam, E. Toubiana: {\it General curvature estimates for stable H-surfaces in $3$-manifolds and applications}. Journal of Differential Geometry, v. 84, p. 623-648, 2010.

\bibitem{Sc} R. Schoen, {\it Estimates for stable minimal surfaces in three-dimensional manifolds}. Seminar on minimal submanifolds, 111–126, Ann. of Math. Stud., 103, Princeton Univ. Press, Princeton, NJ, 1983. 

\bibitem{S.Y2} R. Schoen \& S.T. Yau, {\it Complete three dimensional manifolds with positive Ricci curvature and scalar curvature}. Ann. Math. Stud., 102 (1982), 209–228.


\bibitem{W1} B. White, {\it Complete Surfaces of Finite Total Curvature}. J. Differential Geom. 26 (1987), no. 2, 315-326.

\bibitem{W2} B. White, {\it Correction to "Complete Surfaces of Finite Total Curvature"}. J. Diff. Geom. 28 (1988),
359-360.

\end{thebibliography}
\end{document}